\documentclass[a4paper,11pt]{article}
\usepackage{amsfonts}
\usepackage{amsmath}
\usepackage{amssymb}
\usepackage{hyperref}

\newtheorem{theorem}{Theorem}

\newtheorem{definition}[theorem]{Definition}
\newtheorem{example}[theorem]{Example}

\newtheorem{remark}[theorem]{Remark}

\newenvironment{proof}[1][Proof]{\noindent\textbf{#1.} }{\ \rule{0.5em}{0.5em}}
\input{tcilatex}
\begin{document}

\title{The Generalized Stokes theorem for $\mathbb{R}$-linear forms\\
on Lie algebroids}
\author{Bogdan Balcerzak}
\date{{\small Institute of Mathematics, Technical University of \L \'{o}d%
\'{z}}\\
{\small W\'{o}lcza\'{n}ska 215, 90-924 \L \'{o}d\'{z}, Poland, e-mail:
bogdan.balcerzak@p.lodz.pl}}
\maketitle

\begin{abstract}
The author presents the generalized Stokes theorem for $\mathbb{R}$-linear
forms on Lie algebroids (which can be non-local). We apply the Stokes
formula on forms to prove that two homotopic homomorphisms of Lie algebroids
implies the existence of a chain operator joining their pullback operators.
\end{abstract}

\ \vspace{1cm}

\hspace{-0.6cm}\textbf{Keywords}: Lie algebroid, homomorphisms of Lie
algebroids, homotopic homomorphisms of Lie algebroids, Lie algebroid
cohomology\vspace{0.3cm}

\hspace{-0.6cm}\textbf{Mathematics Subject Classification (2010)}\textsc{: }%
58H05, 17B56, 58A10\vspace{1.6cm}

\section{Introduction}

Some authors, e.g. Evens, Lu, Weinstein in \cite{Evens-Lu-Weinstein},
Crainic, Fernandes in \cite{Crainic-up to homotopy}, \cite%
{Crainic-Fernandes-jets}, proposed the use of some $\mathbb{R}$-linear
connections and $\mathbb{R}$-linear forms to examine some characteristic
classes. In \cite{Evens-Lu-Weinstein} the authors define the modular class
of the Lie algebroid using some $\mathbb{R}$-linear connection, namely the
adjoint representation, and introduce the more general notion of a
representation up to homotopy. Such connections and so-called non-linear
forms were used by Crainic and Fernandes to study secondary characteristic
classes on Lie algebroids in the more general context. Using the definition
of Grabowski-Marmo-Michor it was shown in \cite{Balcerzak} that the modular
class of a base-preserving homomorphism of Lie algebroids is the
Chern-Simons form for a pair of $\mathbb{R}$-linear connections determined
by some distinguished divergences.

It was the motivation to investigate whether the classical Stokes' theorem
extends to $\mathbb{R}$-linear forms. In this text we present the Stokes
formula for $\mathbb{R}$-linear forms on Lie algebroids which are more
general than usual and non-linear forms. The difficulty lies in that we
cannot use a local property for $\mathbb{R}$-linear forms. Moreover, we
formulate suitable results for (linear) differential forms on Lie
algebroids, which was stated by I.~Vaisman in \cite[2010]{Vaisman}. These
generalize the known result for tangent bundles given by Bott \cite{Bott}.
We apply this result to homotopic homomorphism of Lie algebroids giving a
generalization of the result for regular Lie algebroids from \cite%
{Kubarski-invariant}. Namely, we prove that two homotopic homomorphisms of
(arbitrary) Lie algebroids implies the existence of a chain operator joining
their pullback operators.

\section{Forms on Lie Algebroids}

By a \emph{Lie algebroid} we mean a triple $\left( A,\rho _{A},[\![\cdot
,\cdot ]\!]_{A}\right) $, in which $A$ is a real vector bundle over a
manifold $M$, $\rho _{A}:A\rightarrow TM$ is a homomorphism of vector
bundles called an \emph{anchor}, $\left( \Gamma \left( A\right) ,[\![\cdot
,\cdot ]\!]_{A}\right) $ is an $\mathbb{R}$-Lie algebra and the Leibniz
identity%
\begin{equation*}
\lbrack \![a,f\cdot b]\!]_{A}=f\cdot \lbrack \![a,b]\!]_{A}+\rho _{A}\left(
a\right) \left( f\right) \cdot b\ \ \ \ \ \text{for all\ \ \ \ }a,b\in
\Gamma \left( A\right) ,\ f\in C^{\infty }\left( M\right)
\end{equation*}%
holds (\cite{Pradines}). The anchor induces a homomorphism of Lie algebras $%
\limfunc{Sec}\rho _{A}:\Gamma \left( A\right) \rightarrow \mathfrak{X}\left(
M\right) $, $a\mapsto \rho _{A}\circ a$, because the representation $\varrho
:C^{\infty }\left( M\right) \rightarrow \limfunc{End}\nolimits_{C^{\infty
}\left( M\right) }\left( \Gamma \left( A\right) \right) $ given by $\varrho
\left( \nu \right) \left( a\right) =\nu \cdot a$ for all $\nu \in C^{\infty
}\left( M\right) $,$\ a\in \Gamma \left( A\right) $, is faithful (see \cite%
{Herz}). If $\rho _{A}$ is a constant rank (i.e. $\func{Im}\rho _{A}$ is a
constant dimensional and completely integrable distribution), we say that $%
\left( A,\rho _{A},[\![\cdot ,\cdot ]\!]_{A}\right) $ is \emph{regular}. By
a \emph{homomorphism of Lie algebroids} $\left( A,\rho _{A},[\![\cdot ,\cdot
]\!]_{A}\right) $, $\left( B,\rho _{B},[\![\cdot ,\cdot ]\!]_{B}\right) $,
both over the same manifold $M$, we mean a homomorphism of vector bundles $%
\Phi :A\rightarrow B$ over $\limfunc{id}\nolimits_{M}$ such that $\rho
_{B}\circ \Phi =\rho _{A}$ and $\Phi \circ \lbrack \![a,b]\!]_{A}=[\![\Phi
\circ a,\Phi \circ b]\!]_{B}$ for all $a,b\in \Gamma \left( A\right) $. We
say that two Lie algebroids are \emph{isomorphic} if there exists their
homomorphism which is an isomorphism of vector bundles. For more information
about Lie algebroids we refer the reader to \cite{Higgins-Mackenzie}, \cite%
{Mackenzie}, \cite{Kubarski-Lyon}.

Let $\left( A,\rho _{A},[\![\cdot ,\cdot ]\!]_{A}\right) $ be a Lie
algebroid on a manifold $M$. By an $n$-differential form on $A$ we mean a
section $\eta \in \Gamma \left( \bigwedge\nolimits^{n}A^{\ast }\right) $. In
the space $\Omega \left( A\right) =\bigoplus\limits_{n\geq 0}\Gamma \left(
\bigwedge\nolimits^{n}A^{\ast }\right) $ we have the exterior differential
operator $d_{A}$ given by the classical formula%
\begin{eqnarray}
\left( d_{A}\eta \right) \left( a_{1},\ldots ,a_{n+1}\right)
&=&\dsum\limits_{i=1}^{n+1}\left( -1\right) ^{i+1}\left( \rho _{A}\circ
a_{i}\right) \left( \eta \left( a_{1},\ldots \widehat{i}\ldots ,a_{n}\right)
\right)  \label{d_A} \\
&&+\dsum\limits_{i<j}\left( -1\right) ^{i+j}\eta \left( \lbrack
\![a_{i},a_{j}]\!]_{A},a_{1},\ldots \widehat{i}\ldots \widehat{j}\ldots
,a_{n+1}\right)  \notag
\end{eqnarray}%
for all $n\geq 1$, $\eta \in \Omega ^{n}\left( A\right) $, $a_{1},\ldots
,a_{n+1}\in \Gamma \left( A\right) $ and $d_{A}\left( f\right) \left(
a\right) =\left( \rho _{A}\circ a\right) \left( f\right) $ for $f\in \Omega
^{0}\left( A\right) =C^{\infty }\left( M\right) $, $a\in \Gamma \left(
A\right) $. The cohomology space of the complex $\left( \Omega _{C^{\infty
}\left( M\right) }\left( A\right) ,d_{A}\right) $ is called the \emph{%
cohomology space of the Lie algebroid} $A$, and is denoted by $H^{\bullet
}\left( A\right) $.

We will extend $d_{A}$ to $\mathbb{R}$-linear forms. Let $\left( A,\rho
_{A},[\![\cdot ,\cdot ]\!]_{A}\right) $ be a Lie algebroid on a manifold $M$%
. An $\mathbb{R}$-multilinear, antisymmetric map%
\begin{equation*}
\omega :\Gamma \left( A\right) \times \cdots \times \Gamma \left( A\right)
\longrightarrow C^{\infty }\left( M\right)
\end{equation*}%
is called an $\mathbb{R}$\emph{-linear form} on $A$. The space of all $%
\mathbb{R}$-linear $n$-forms on $A$ will be denoted by $\mathcal{A}lt_{%
\mathbb{R}}^{n}\left( \Gamma \left( A\right) ;C^{\infty }\left( M\right)
\right) $, and the space of $\mathbb{R}$-linear forms on $A$ by $\mathcal{A}%
lt_{\mathbb{R}}^{\bullet }\left( \Gamma \left( A\right) ;C^{\infty }\left(
M\right) \right) =\bigoplus\limits_{n\geq 0}\mathcal{A}lt_{\mathbb{R}%
}^{n}\left( \Gamma \left( A\right) ;C^{\infty }\left( M\right) \right) $,
where $\mathcal{A}lt_{\mathbb{R}}^{0}\left( \Gamma \left( A\right)
;C^{\infty }\left( M\right) \right) =C^{\infty }\left( M\right) $. We extend
the exterior multiplication of differential forms on the Lie algebroid to
the space $\mathcal{A}lt_{\mathbb{R}}^{\bullet }\left( \Gamma \left(
A\right) ;C^{\infty }\left( M\right) \right) $, obtaining the structure of
an algebra and extend $d_{A}$ to a differential operator%
\begin{equation*}
d_{A,\mathbb{R}}:\mathcal{A}lt_{\mathbb{R}}^{\bullet }\left( \Gamma \left(
A\right) ;C^{\infty }\left( M\right) \right) \longrightarrow \mathcal{A}lt_{%
\mathbb{R}}^{\bullet +1}\left( \Gamma \left( A\right) ;C^{\infty }\left(
M\right) \right)
\end{equation*}%
by the same formula as in (\ref{d_A}).

Observe that for a Lie algebroid $\left( A,\rho _{A},[\![\cdot ,\cdot
]\!]_{A}\right) $ over a compact orientable manifold $M$ with a volume form $%
\Omega $ every form $\eta \in \Omega ^{n}\left( A\right) $ on $A$ defines an 
$\mathbb{R}$-linear form%
\begin{eqnarray*}
\widetilde{\eta } &\in &\mathcal{A}lt_{\mathbb{R}}^{n}\left( \Gamma \left(
A\right) ;C^{\infty }\left( M\right) \right) , \\
\widetilde{\eta }\left( a_{1},\ldots ,a_{n}\right) &=&\int_{M}\eta \left(
a_{1},\ldots ,a_{n}\right) \Omega ,
\end{eqnarray*}%
$a_{1},\ldots ,a_{n}\in \Gamma \left( A\right) $, which is, in general,
nonlocal.

\section{A Few Words About Homomorphisms of Lie Algebroids and the Pullback
of Forms}

Let $\left( A,\rho _{A},[\![\cdot ,\cdot ]\!]_{A}\right) $ and $\left(
B,\rho _{B},[\![\cdot ,\cdot ]\!]_{B}\right) $ be Lie algebroids over the
same manifold $M$, $\Phi :A\rightarrow B$ a homomorphism of Lie algebroids
over the identity. We define an operator of zero degree as follows:%
\begin{equation*}
\Phi ^{\ast }:\mathcal{A}lt_{\mathbb{R}}^{\bullet }\left( \Gamma \left(
B\right) ;C^{\infty }\left( M\right) \right) \longrightarrow \mathcal{A}lt_{%
\mathbb{R}}^{\bullet }\left( \Gamma \left( A\right) ;C^{\infty }\left(
M\right) \right) ,
\end{equation*}%
\begin{equation*}
\Phi ^{\ast }\left( \omega \right) \left( a_{1},\ldots ,a_{n}\right) =\omega
\left( \Phi \circ a_{1},\ldots ,\Phi \circ a_{n}\right)
\end{equation*}%
for all $\omega \in \mathcal{A}lt_{\mathbb{R}}^{n}\left( \Gamma \left(
B\right) ;C^{\infty }\left( M\right) \right) $, $n\geq 1$, $a_{1},\ldots
,a_{n}\in \Gamma \left( A\right) $, and $\Phi ^{\ast }\left( h\right)
=h\circ \Phi $ for $h\in C^{\infty }\left( M\right) $; the form $\Phi ^{\ast
}\left( \omega \right) $ is called the \emph{pullback} of $\omega $ via $%
\Phi $. Since $\rho _{B}\circ \Phi =\rho _{A}$ and $\Phi $ preserves
brackets, we see that%
\begin{equation}
\Phi ^{\ast }\circ d_{\mathbb{R}}^{B}=d_{\mathbb{R}}^{A}\circ \Phi ^{\ast }.
\label{przem_cofania_z_rozniczka}
\end{equation}%
Now we recall the definition of a homomorphism of Lie algebroids (Higgins
and Mackenzie, \cite{Higgins-Mackenzie}) which extends the notion of a Lie
algebroid homomorphism over the identity map and the definition of a
pullback of ($C^{\infty }\left( M\right) $-linear) \text{forms of the Lie
algebroid (Kubarski, \cite{Kubarski-invariant}).}

\begin{definition}
\emph{By a} homomorphism%
\begin{equation*}
\Phi :\left( A,\rho _{A},[\![\cdot ,\cdot ]\!]_{A}\right) \longrightarrow
\left( B,\rho _{B},[\![\cdot ,\cdot ]\!]_{B}\right)
\end{equation*}%
\emph{of Lie algebroids }$\left( A,\rho _{A},[\![\cdot ,\cdot
]\!]_{A}\right) $\emph{\ and }$\left( B,\rho _{B},[\![\cdot ,\cdot
]\!]_{B}\right) $\emph{, where the first is over a manifold }$M$\emph{, the
second over a manifold }$N$\emph{, we mean a homomorphism of vector bundles }%
$\Phi :A\rightarrow B$\emph{\ over }$f:M\rightarrow N$\emph{\ such that }$%
\rho _{B}\circ \Psi =f_{\ast }\circ \rho _{A}$\emph{, and for all
cross-sections }$a,b\in \Gamma \left( A\right) $\emph{\ with }$\Phi $\emph{%
-decompositions }$\Phi \circ a=\sum\nolimits_{i}f_{a}^{i}\left( \sigma
_{i}\circ f\right) $\emph{, }$\Phi \circ b=\sum\nolimits_{j}f_{b}^{j}\left(
\varepsilon _{j}\circ f\right) $\emph{, where }$f_{a}^{i},f_{b}^{j}\in
C^{\infty }\left( M\right) $\emph{, }$\sigma _{i},\varepsilon _{j}\in \Gamma
\left( B\right) $\emph{, we have}%
\begin{eqnarray*}
\Phi \circ \lbrack \![a,b]\!]_{A}
&=&\sum\limits_{i,j}f_{a}^{i}f_{b}^{j}\left( [\![\sigma _{i},\varepsilon
_{j}]\!]_{A}\circ f\right) +\sum\limits_{j}\left( \rho _{A}\circ a\right)
\left( f_{b}^{j}\right) \cdot \left( \varepsilon _{j}\circ f\right) \\
&&-\sum\limits_{i}\left( \rho _{B}\circ b\right) \left( f_{b}^{i}\right)
\cdot \left( \sigma _{i}\circ f\right) .
\end{eqnarray*}
\end{definition}

\begin{definition}
\emph{Let }$\Phi $\emph{\ be a homomorphism of Lie algebroids }$\left(
A,\rho _{A},[\![\cdot ,\cdot ]\!]_{A}\right) $\emph{\ and }$\left( B,\rho
_{B},[\![\cdot ,\cdot ]\!]_{B}\right) $\emph{\ over }$f:M\rightarrow N$\emph{%
. A }pullback\emph{\ of a form }$\eta \in \Omega ^{n}\left( B\right) $\emph{%
\ is a form }$\Phi ^{\ast }\eta \in \Omega ^{n}\left( A\right) $\emph{\ such
that}%
\begin{equation*}
\Phi ^{\ast }\eta \left( x;\nu _{1}\wedge \ldots \wedge \nu _{n}\right)
=\eta \left( f\left( x\right) ;\Phi \nu _{1}\wedge \ldots \wedge \Phi \nu
_{n}\right)
\end{equation*}%
\emph{for all }$x\in M$\emph{, }$\nu _{1},\ldots ,\nu _{n}\in A_{|x}$\emph{.}
\end{definition}

A homomorphism $\Phi :\left( A,\rho _{A},[\![\cdot ,\cdot ]\!]_{A}\right)
\rightarrow \left( B,\rho _{B},[\![\cdot ,\cdot ]\!]_{B}\right) $ of Lie
algebroids induces a homomorphism of algebras $\Phi ^{\ast }:\Omega \left(
B\right) \rightarrow \Omega \left( A\right) $ such that $d_{A}\circ \Phi
^{\ast }=\Phi ^{\ast }\circ d_{B}$ (\cite{Kubarski-invariant}). Therefore, $%
\Phi $ defines the homomorphism $\Phi ^{\#}:H\left( B\right) \rightarrow
H\left( A\right) $ on cohomologies. Moreover, we see that for two
homomorphisms of Lie algebroids $\Psi :A\rightarrow B$ and $\Phi
:B\rightarrow C$ (over $f:M\rightarrow N$ and $g:N\rightarrow P$,
respectively) holds%
\begin{equation*}
\left( \Phi \circ \Psi \right) ^{\ast }=\Psi ^{\ast }\circ \Phi ^{\ast }.
\end{equation*}

\begin{theorem}
\label{theorem_the_inverse_image}\emph{(\cite%
{Kosmann-Schwarzbach-Laurent-Gengoux-Weinstein})} Let $\left( B,\rho
_{B},[\![\cdot ,\cdot ]\!]_{B}\right) $ be a Lie algebroid over a manifold $%
N $, $f:M\rightarrow N$ be a smooth map such that%
\begin{equation*}
f^{\;\wedge }\hspace{-0.1cm}\left( B\right) =\left\{ \left( \gamma ,b\right)
\in TM\times B:f_{\ast }\gamma =\rho _{B}\left( b\right) \right\}
\end{equation*}%
is a vector subbundle of $TM\oplus f^{\ast }B$ over $M$ \emph{(}$f$ \emph{is
called then} admissible\emph{)}. Then $f^{\;\wedge }\hspace{-0.1cm}\left(
B\right) $ has a Lie algebroid structure with the projection to the first
factor as an anchor and the bracket $[\![\cdot ,\cdot ]\!]^{\wedge }$
defined in the following way: for $\left( X,\overline{\zeta }\right)
,~\left( Y,\overline{\sigma }\right) \in \Gamma \left( f^{\;\wedge }\hspace{%
-0.1cm}\left( B\right) \right) $, where $X,Y\in \limfunc{Der}\left(
C^{\infty }\left( M\right) \right) $ and $\overline{\zeta },\overline{\sigma 
}\in \Gamma \left( f^{\ast }B\right) $ there exist $n\in 
\mathbb{N}
$, sections $\zeta ^{1},\ldots ,\zeta ^{n}$, $\sigma ^{1},\ldots ,\sigma
^{n} $ of $B$ and $f^{1},\ldots ,f^{n},g^{1},\ldots ,g^{n}\in C^{\infty
}\left( M\right) $, such that locally (on an open set $U\subset M$) $%
\overline{\zeta }$, $\overline{\sigma }$ are respectively of the form $%
\sum\nolimits_{p}f^{p}\cdot \left( \eta ^{p}\circ f\right) $ and $%
\sum\nolimits_{q}g^{q}\cdot \left( \sigma ^{q}\circ f\right) $, and we
define $[\![\left( X,\overline{\zeta }\right) ,~\left( Y,\overline{\sigma }%
\right) ]\!]^{\wedge }$ on $U$ by 
\begin{equation*}
\hspace{-0.2cm}\left( \left[ X,Y\right] ,\sum\limits_{p,q}f^{p}g^{q}\left(
[\![\zeta ^{p},\sigma ^{q}]\!]_{B}\circ f\right) +\sum\limits_{q}X\hspace{%
-0.1cm}\left( g^{q}\right) \cdot \left( \sigma ^{q}\circ f\right)
-\sum\limits_{p}Y\hspace{-0.1cm}\left( f^{p}\right) \cdot \left( \zeta
^{p}\circ f\right) \right) _{\hspace{-0.1cm}|U}\hspace{-0.3cm}.
\end{equation*}
\end{theorem}

\begin{definition}
\emph{Let }$\left( B,\rho _{B},[\![\cdot ,\cdot ]\!]_{B}\right) $\emph{\ be
a Lie algebroid over a manifold }$N$\emph{, }$f:M\rightarrow N$\emph{\ be a
smooth admissible map. A Lie algebroid }$\left( f^{\;\wedge }\hspace{-0.1cm}%
\left( B\right) ,\limfunc{pr}\nolimits_{1},[\![\cdot ,\cdot ]\!]^{\wedge
}\right) $\emph{\ described in the above theorem is called the }inverse image%
\emph{\ of }$B$\emph{\ via }$f$\emph{.}
\end{definition}

\begin{example}
\label{example_admissible maps}\emph{Consider a regular Lie algebroid }$%
\left( B,\rho _{B},[\![\cdot ,\cdot ]\!]_{B}\right) $\emph{\ over a manifold 
}$N$\emph{, a manifold }$M$\emph{, a subbundle }$F$\emph{\ of }$TM$\emph{.
Any smooth map }$f:M\rightarrow N$\emph{\ such that }$f_{\ast }\left(
F\right) \subset \func{Im}\rho $ \emph{is admissible. Indeed, }$f^{\;\wedge }%
\hspace{-0.1cm}\left( B\right) $\emph{\ is a vector bundle, because for
every }$x\in M$\emph{, }$f^{\;\wedge }\hspace{-0.1cm}\left( B\right)
_{|x}=\ker G_{|x}$\emph{\ where }$G:F\oplus f^{\ast }B\rightarrow \func{Im}%
\rho _{B}$\emph{\ is a morphism of vector bundles over }$f$\emph{\ given by }%
$G_{x}\left( \tau ,\beta \right) =f_{\ast }\left( \tau \right) -\rho
_{B}\left( \beta \right) $\emph{\ for all }$x\in M$\emph{, }$\tau \in F_{|x}$%
\emph{, }$\beta \in B_{|f\left( x\right) }$\emph{\ and the function }$M\ni
x\mapsto \dim \ker G_{|x}\in \mathbb{Z}$\emph{\ is constant (see also \cite%
{Kubarski-Lyon}). In particular, if }$\left( B,\rho _{B},[\![\cdot ,\cdot
]\!]_{B}\right) $\emph{\ is a Lie algebroid with a surjective anchor }$\rho
_{B}:B\rightarrow TN$\emph{\ (then we say that }$B$\emph{\ is }transitive%
\emph{), then any smooth mapping }$f:M\rightarrow N$\ \emph{is admissible.
Moreover, any surjective submersion is admissible.}
\end{example}

\begin{example}
\emph{\cite{Kosmann-Schwarzbach-Laurent-Gengoux-Weinstein} Any map }$%
f:M\rightarrow N$\emph{\ transverse to a given Lie algebroid }$\left( B,\rho
_{B},[\![\cdot ,\cdot ]\!]_{B}\right) $\emph{\ over }$M$\emph{\ (i.e. }$%
df\left( T_{x}M\right) +\rho _{B}\left( B_{f\left( x\right) }\right)
=T_{f\left( x\right) }N$\emph{\ for all }$x\in M$\emph{) is admissible.}
\end{example}

Consider a homomorphism $\Phi :\left( A,\rho _{A},[\![\cdot ,\cdot
]\!]_{A}\right) \rightarrow \left( B,\rho _{B},[\![\cdot ,\cdot
]\!]_{B}\right) $ of regular Lie algebroids over $f:M\rightarrow N$. From
the above example we see that $f$\ is admissible and $\left( f^{\;\wedge }%
\hspace{-0.1cm}\left( B\right) ,\limfunc{pr}_{1},[\![\cdot ,\cdot
]\!]^{\wedge }\right) $ is a Lie algebroid (see also \cite{Kubarski-Lyon})
with the projection on the first factor as an anchor and the bracket $%
[\![\cdot ,\cdot ]\!]^{\wedge }$ on the $C^{\infty }\left( M\right) $-module 
$\Gamma \left( f^{\;\wedge }\hspace{-0.1cm}\left( B\right) \right) \subset 
\mathfrak{X}\left( M\right) \times \Gamma \left( \func{pr}_{2}^{\ast
}A\right) \cong \mathfrak{X}\left( M\right) \times C^{\infty }\left(
M\right) \otimes _{C^{\infty }\left( M\right) }\Gamma \left( B\right) $; see 
\cite{Higgins-Mackenzie}. Then $\Phi $ can be written as a composition of
two homomorphisms of Lie algebroids%
\begin{equation*}
\Phi =\chi \circ \overline{\Phi }
\end{equation*}%
where $\overline{\Phi }:A\rightarrow f^{\;\wedge }\hspace{-0.1cm}\left(
B\right) $ is a (base-preserving) homomorphism of $A$ and the inverse--image
of $B$ via $f$ (see \cite{Kubarski-Lyon}, \cite{Kubarski-invariant}) given
by $\alpha \mapsto \left( \rho _{A}\left( \alpha \right) ,\Phi \left( \alpha
\right) \right) $, and $\chi :f^{\;\wedge }\hspace{-0.1cm}\left( B\right)
\rightarrow B$ is the projection to the second factor. Hence, the pullback
operator $\Phi ^{\ast }:\Omega \left( B\right) \rightarrow \Omega \left(
A\right) $ can be represented as a composition $\Phi ^{\ast }=\overline{\Phi 
}^{\ast }\circ \chi ^{\ast }$ where $\chi ^{\ast }:\Omega \left( B\right)
\rightarrow \Omega \left( f^{\wedge }\hspace{-0.1cm}\left( B\right) \right) 
\mathcal{\ }$and%
\begin{equation*}
\chi ^{\ast }\left( \omega \right) \left( c_{1},\ldots ,c_{n}\right)
=\dsum\limits_{i_{1},\ldots ,i_{n}}f^{i_{1}}\ldots f^{i_{n}}\left( \omega
\left( \xi ^{i_{1}},\ldots ,\xi ^{i_{n}}\right) \circ f\right)
\end{equation*}%
for all $\omega \in \Omega ^{n}\left( B\right) $, $n\geq 1$, and for all $%
c_{1}=\left( X_{1},\tsum\limits_{i_{1}}f^{i_{1}}\otimes \xi ^{i_{1}}\right) $%
,$\ldots $, $c_{n}=\left( X_{n},\tsum\limits_{i_{n}}f^{i_{n}}\otimes \xi
^{i_{n}}\right) $ from $\mathfrak{X}\left( M\right) \times C^{\infty }\left(
M\right) \otimes _{C^{\infty }\left( N\right) }\Gamma \left( B\right) $.
Moreover, $\chi ^{\ast }\left( g\right) =g\circ f$ for all $g\in C^{\infty
}\left( N\right) $.

We recall the definition of the Cartesian product of two Lie algebroids from 
\cite{Kubarski-invariant}.

\begin{definition}
\emph{The }Cartesian product of two Lie algebroids\emph{\ }$\left( A,\rho
_{A},[\![\cdot ,\cdot ]\!]_{A}\right) $\emph{\ and }$\left( B,\rho
_{B},[\![\cdot ,\cdot ]\!]_{B}\right) $\emph{\ over manifolds }$M$\emph{\
and }$N$\emph{, respectively, is the Lie algebroid}%
\begin{equation*}
\left( A\times B,\rho _{A}\times \rho _{B},[\![\cdot ,\cdot ]\!]_{A\times
B}\right)
\end{equation*}%
\emph{over }$M\times N$\emph{\ with the bracket }$[\![\cdot ,\cdot
]\!]_{A\times B}$\emph{\ in }$\Gamma \left( A\times B\right) $\emph{, given
in such a way that for all }$\overline{\sigma }=\left( \overline{\sigma }%
^{1},\overline{\sigma }^{2}\right) $\emph{, }$\overline{\eta }=\left( 
\overline{\eta }^{1},\overline{\eta }^{2}\right) \in \Gamma \left( A\times
B\right) $\emph{, }%
\begin{equation*}
\lbrack \![\overline{\sigma },\overline{\eta }]\!]_{A\times B}=\left( [\![%
\overline{\sigma },\overline{\eta }]\!]^{1},[\![\overline{\sigma },\overline{%
\eta }]\!]^{2}\right) ,
\end{equation*}%
\emph{where for all} $\left( x,y\right) \in M\times N,$\emph{\ }$[\![%
\overline{\sigma },\overline{\eta }]\!]^{1}\left( x,y\right) $ \emph{and }$%
[\![\overline{\sigma },\overline{\eta }]\!]^{2}\left( x,y\right) $ \emph{are
equal to}%
\begin{gather*}
\lbrack \![\overline{\sigma }^{1}\left( \cdot ,y\right) ,\overline{\eta }%
^{1}\left( \cdot ,y\right) ]\!]_{A}\left( x\right) +\left( \rho _{B}\circ 
\overline{\sigma }^{2}\right) _{\left( x,y\right) }\left( \overline{\eta }%
^{1}\left( x,\cdot \right) \right) -\left( \rho _{B}\circ \overline{\eta }%
^{2}\right) _{\left( x,y\right) }\left( \overline{\sigma }^{1}\left( x,\cdot
\right) \right) , \\
\lbrack \![\overline{\sigma }^{2}\left( x,\cdot \right) ,\overline{\eta }%
^{2}\left( x,\cdot \right) ]\!]_{B}\left( y\right) +\left( \rho _{A}\circ 
\overline{\sigma }^{1}\right) _{\left( x,y\right) }\left( \overline{\eta }%
^{2}\left( \cdot ,y\right) \right) -\left( \rho _{A}\circ \overline{\eta }%
^{1}\right) _{\left( x,y\right) }\left( \overline{\sigma }^{2}\left( \cdot
,y\right) \right) ,
\end{gather*}%
\emph{respectively.}
\end{definition}

\newpage

\section{The Generalized Stokes Theorem on Lie Algebroids}

In this section we will prove the Stokes formula for Lie algebroids and $%
\mathbb{R}$-linear (not necessarily local) forms, which is a generalization
of the known formula from \cite{Bott}.

Let $\left( A,\rho _{A},[\![\cdot ,\cdot ]\!]_{A}\right) $ be a Lie
algebroid on a manifold $M$. For every natural $k$, let $\func{pr}_{2}:%
\mathbb{R}^{k}\times M\rightarrow M$\ be a projection on the second factor
and 
\begin{equation*}
\Delta ^{k}=\left\{ \left( t_{1},...,t_{k}\right) \in \mathbb{R}%
^{k};\;\;\;\forall i\;\;t_{i}\geq 0\,,\;\;\sum\nolimits_{i=1}^{k}t_{i}\leq
1\right\}
\end{equation*}%
the \emph{standard }$k$\emph{-simplex} in $\mathbb{R}^{k}$. Additionally we
set the \emph{standard }$0$\emph{-simplex} as $\Delta ^{0}=\left\{ 0\right\} 
$.

Recall that $C^{\infty }\left( \mathbb{R}\times M\right) $-modules $\Gamma
\left( \func{pr}_{2}^{\ast }A\right) $ and $C^{\infty }\left( \mathbb{R}%
^{k}\times M\right) \otimes _{C^{\infty }\left( M\right) }\Gamma \left(
A\right) $ are isomorphic (see \cite{Higgins-Mackenzie}) and 
\begin{eqnarray*}
\func{pr}_{2}^{\;\wedge }\hspace{-0.1cm}\left( A\right) &=&\left\{ \left(
\gamma ,w\right) \in T\left( \mathbb{R}^{k}\times M\right) \times A:\left( 
\func{pr}_{2}\right) _{\ast }\gamma =\rho _{A}\left( w\right) \right\} \\
&\subset &T\left( \mathbb{R}^{k}\times M\right) \oplus \func{pr}_{2}^{\ast }A
\end{eqnarray*}%
is a Lie algebroid over $\mathbb{R}^{k}\times M$ with the projection on the
first factor as an anchor and the bracket $[\![\cdot ,\cdot ]\!]^{\wedge }$
given in Theorem \ref{theorem_the_inverse_image} on the module $\Gamma
\left( \func{pr}_{2}^{\;\wedge }\hspace{-0.1cm}\left( A\right) \right)
\subset \mathfrak{X}\left( \mathbb{R}^{k}\times M\right) \times \Gamma
\left( \func{pr}_{2}^{\ast }A\right) \cong \mathfrak{X}\left( \mathbb{R}%
^{k}\times M\right) \times C^{\infty }\left( \mathbb{R}^{k}\times M\right)
\otimes _{C^{\infty }\left( M\right) }\Gamma \left( A\right) $. The map%
\begin{equation*}
\Psi :T\mathbb{R}^{k}\times A\longrightarrow \func{pr}_{2}^{\;\wedge }%
\hspace{-0.1cm}\left( A\right) ,\ \ \ \left( u,w\right) \longmapsto \left(
u,\rho _{A}\left( w\right) ,w\right)
\end{equation*}%
is an isomorphism of Lie algebroids $T\mathbb{R}^{k}\times A$ and\ $\func{pr}%
_{2}^{\;\wedge }\hspace{-0.1cm}\left( A\right) $. In view of the
identification $\Gamma \left( T\mathbb{R}^{k}\times A\right) $\newline
$\simeq \Gamma \left( \func{pr}_{2}^{\;\wedge }\hspace{-0.1cm}\left(
A\right) \right) $ as $C^{\infty }\left( \mathbb{R}^{k}\times M\right) $%
-modules, we will treat $\Gamma \left( T\mathbb{R}^{k}\times A\right) $ as a 
$C^{\infty }\left( \mathbb{R}^{k}\times M\right) $-submodule of 
\begin{equation*}
\mathfrak{X}\left( \mathbb{R}^{k}\times M\right) \times \left( C^{\infty
}\left( \mathbb{R}^{k}\times M\right) \otimes _{C^{\infty }\left( M\right)
}\Gamma \left( A\right) \right) .
\end{equation*}%
The cross-section $\left( 0,\xi \right) $ of a vector bundle $T\mathbb{R}%
^{k}\times A$ will be simply denoted by $\xi $ and $(\frac{\partial }{%
\partial t^{j}},0)$ by $\frac{\partial }{\partial t^{j}}$. Define%
\begin{equation*}
\dint\nolimits_{\Delta ^{k}}:\mathcal{A}lt_{\mathbb{R}}^{\bullet }\left(
\Gamma \left( T\mathbb{R}^{k}\times A\right) ;C^{\infty }\left( \mathbb{R}%
^{k}\times M\right) \right) \longrightarrow \mathcal{A}lt_{\mathbb{R}%
}^{\bullet -k}\left( \Gamma \left( A\right) ;C^{\infty }\left( M\right)
\right) ,
\end{equation*}%
\begin{equation*}
\left( \int\nolimits_{\Delta ^{k}}\omega \right) \left( a_{1},\ldots
,a_{n-k}\right) =\int\nolimits_{\Delta ^{k}}\omega \left( \frac{\partial }{%
\partial t^{1}},\ldots ,\frac{\partial }{\partial t^{k}},a_{1},\ldots
,a_{n-k}\right) _{|\left( t_{1},...,t_{k},\bullet \right) }dt_{1}...dt_{k}
\end{equation*}%
for all $n\geq 1$, $1\leq k\leq n$, $\omega \in \mathcal{A}lt_{\mathbb{R}%
}^{n}\left( \Gamma \left( T\mathbb{R}^{k}\times A\right) ;C^{\infty }\left( 
\mathbb{R}^{k}\times M\right) \right) $,\ $a_{1}$,$...$,$a_{n-k}\in \Gamma
\left( A\right) $\ and%
\begin{equation*}
\left( \int\nolimits_{\Delta ^{0}}\omega \right) \left( a_{1},\ldots
,a_{n}\right) =\iota _{0}^{\ast }\left( \omega \left( 0\times a_{1},\ldots
,0\times a_{n}\right) \right) ,\ \ \ \int\nolimits_{\Delta ^{0}}f=\iota
_{0}^{\ast }f
\end{equation*}%
for all for $n\geq 1$, $\omega \in \mathcal{A}lt_{\mathbb{R}}^{n}\left(
\Gamma \left( T\mathbb{R}^{k}\times A\right) ;C^{\infty }\left( M\right)
\right) $, $a_{1},\ldots ,a_{n}\in \Gamma \left( A\right) $, $f\in C^{\infty
}\left( \mathbb{R}^{k}\times M\right) $ and where $\iota _{0}:M\rightarrow
\Delta ^{0}\times M$ is an inclusion defined by $\iota _{0}\left( x\right)
=\left( 0,x\right) $.

\begin{theorem}
\label{Stokes theorem for R-linear forms}\emph{(The Stokes theorem for }$%
\mathbb{R}$\emph{-linear forms) }For every $k\in \mathbb{N}$,%
\begin{equation}
\int\nolimits_{\Delta ^{k}}\circ \,d_{T\mathbb{R}^{k}\times A,\mathbb{R}%
}+\left( -1\right) ^{k+1}d_{A,\mathbb{R}}\circ \int\nolimits_{\Delta
^{k}}=\dsum\nolimits_{j=0}^{k}\left( -1\right) ^{j}\int\nolimits_{\Delta
^{k-1}}\circ \,\left( d\sigma _{j}^{k-1}\times \func{id}_{A}\right) ^{\ast },
\label{Stokes}
\end{equation}%
where $\sigma _{j}^{k}:\mathbb{R}^{k}\rightarrow \mathbb{R}^{k+1}$\ for $%
0\leq j\leq k+1$\ are functions defined by $\sigma _{0}^{0}\left( 0\right)
=1 $, $\sigma _{1}^{0}\left( 0\right) =0$, and for $\left( t_{1},\ldots
,t_{k}\right) \in \mathbb{R}^{k}$ by%
\begin{eqnarray*}
\sigma _{0}^{k}\left( t_{1},\ldots ,t_{k}\right) &=&\left(
1-\dsum\nolimits_{i=1}^{k}t_{i},t_{1},\ldots ,t_{k}\right) , \\
\sigma _{j}^{k}\left( t_{1},\ldots ,t_{k}\right) &=&\left( t_{1},\ldots
,t_{j-1},0,t_{j},\ldots ,t_{k}\right) ,\;\;1\leq j\leq k+1,
\end{eqnarray*}%
and where $\left( \left( \int\nolimits_{\Delta ^{k-1}}\circ \,\left( d\sigma
_{j}^{k-1}\times \func{id}_{A}\right) ^{\ast }\right) \omega \right) \left(
a_{1},\ldots ,a_{n-k+1}\right) $ is, by definition, equal to%
\begin{equation*}
\int\nolimits_{\Delta ^{k-1}}\omega \left( d\sigma _{j}^{k-1}\left( \frac{%
\partial }{\partial t^{1}}\right) ,\ldots ,d\sigma _{j}^{k-1}\left( \frac{%
\partial }{\partial t^{k-1}}\right) ,a_{1},\ldots ,a_{n-k+1}\right)
_{|\left( t_{1},\ldots ,t_{k-1},\bullet \right) }dt_{1}\ldots dt_{k-1}
\end{equation*}%
and 
\begin{equation*}
\left( \left( \int\nolimits_{\Delta ^{0}}\circ \,\left( d\sigma
_{j}^{0}\times \func{id}_{A}\right) ^{\ast }\right) \omega \right) \left(
a_{1},\ldots ,a_{n}\right) =\left( \sigma _{j}^{0}\times \func{id}_{M}\circ
\iota _{0}\right) ^{\ast }\left( \omega \left( a_{1},\ldots ,a_{n}\right)
\right)
\end{equation*}%
for all $k\geq 2$, $j\in \left\{ 0,1\right\} $, $\omega \in \mathcal{A}lt_{%
\mathbb{R}}^{n}\left( \Gamma \left( T\mathbb{R}^{k}\times A\right)
;C^{\infty }\left( \mathbb{R}^{k}\times M\right) \right) $, $a_{1},\ldots
,a_{n}\in \Gamma \left( A\right) $.
\end{theorem}

\begin{proof}
In the proof of the theorem we make direct calculations because we can not
use local properties of forms as in the classical Stokes' theorem. Changes
of variables via suitable diffeomorphisms have been used. We can consider
two cases: $k=1$ and $k\geq 2$. The first case is left to the reader (we
have here straightforward calculations).

Let $k\geq 2$ and $n\geq k$ be natural numbers, $\Omega \in \mathcal{A}_{%
\mathbb{R}}^{n}\left( T\mathbb{R}^{k}\times A\right) $,\ $%
a_{0},\,...,\,a_{n-k}\in \Gamma \left( A\right) $, $\left( \;\tilde{t}%
^{1},\ldots ,\tilde{t}\,^{k}\right) =\func{id}_{\mathbb{R}^{k}}\;$be the
identity map on $\mathbb{R}^{k}$, \thinspace $t=\left( t_{1},\ldots
,t_{k}\right) \in \mathbb{R}^{k}$. Because of $[\![\frac{\partial }{\partial
\,\tilde{t}^{i}},a_{s}]\!]=0$, $[\![\frac{\partial }{\partial \,\tilde{t}%
^{\,p}},\frac{\partial }{\partial \,\tilde{t}^{\,q}}]\!]=0$, observe that

\begin{multline*}
\left( d_{T\mathbb{R}^{k}\times A,\mathbb{R}}\Omega \right) \left( \frac{%
\partial }{\partial \,\tilde{t}^{1}},\ldots ,\frac{\partial }{\partial 
\tilde{t}^{k}},a_{0},\ldots ,a_{n-k}\right) \\
\hspace{3cm}=\sum\nolimits_{j=1}^{k}\left( -1\right) ^{j+1}\frac{\partial }{%
\partial \,\tilde{t}^{\,j}}\left( \Omega \left( \frac{\partial }{\partial \,%
\tilde{t}^{1}},\ldots \widehat{j}\ldots ,\frac{\partial }{\partial \,\tilde{t%
}^{k}},a_{0},\ldots ,a_{n-k}\right) \right) \\
\ \hspace{4cm}+\sum\nolimits_{i=0}^{n-k}\left( -1\right) ^{k+i}\left( \rho
_{A}\circ a_{i}\right) \left( \Omega \left( \frac{\partial }{\partial \,%
\tilde{t}^{1}},\ldots ,\frac{\partial }{\partial \,\tilde{t}^{k}}%
,a_{0},\ldots \widehat{i}\ldots ,a_{n-k}\right) \right) \\
\ \hspace{3cm}{+\sum\nolimits_{i<j}\left( -1\right) ^{k+i+j}\Omega \left( 
\frac{\partial }{\partial \,\tilde{t}^{1}},\ldots ,\frac{\partial }{\partial
\,\tilde{t}^{k}},[\![a_{i},a_{j}]\!],a_{0},\ldots \widehat{i}\ldots \widehat{%
j}\ldots ,a_{n-k}\right) .}
\end{multline*}%
Furthermore,%
\begin{align*}
& d_{A,\mathbb{R}}\left( \int\nolimits_{\Delta ^{k}}\Omega \right) \left(
a_{0},\ldots ,a_{n-k}\right) \\
& =\int\nolimits_{\Delta ^{k}}\left( \sum\nolimits_{i=0}^{n-k}\left(
-1\right) ^{i}\left( \rho _{A}\circ a_{i}\right) \left( \Omega \left( \frac{%
\partial }{\partial \,\tilde{t}^{1}},\ldots ,\frac{\partial }{\partial \,%
\tilde{t}^{k}},a_{0},\ldots \widehat{i}\ldots ,a_{n-k}\right) \right)
\right) _{|\left( t,\bullet \right) }dt_{1}\ldots dt_{k} \\
& +\int\nolimits_{\Delta ^{k}}\left( \sum\nolimits_{i<j}\left( -1\right)
^{i+j}\Omega \left( \frac{\partial }{\partial \,\tilde{t}^{1}},\ldots ,\frac{%
\partial }{\partial \,\tilde{t}^{k}},[\![a_{i},a_{j}]\!],a_{0},\ldots 
\widehat{i}\ldots \widehat{j}\ldots ,a_{n-k}\right) \right) _{|\left(
t,\bullet \right) }\hspace{-0.2cm}dt_{1}\ldots dt_{k}\,.
\end{align*}%
Hence, putting%
\begin{equation*}
\Delta _{\left( \,j\right) }^{k-1}=\left\{ \left( t_{1},\ldots \widehat{j}%
\ldots ,t_{k}\right) \in \mathbb{R}^{k-1}\,;\;\;\forall i\;\;t_{i}\geq
0\,,\;\;t_{1}+\cdots \widehat{j}\cdots +t_{k}\leq 1\right\} ,
\end{equation*}%
we obtain%
\begin{align*}
& \left( \left( \int\nolimits_{\Delta ^{k}}\circ ~d_{T\mathbb{R}^{k}\times A,%
\mathbb{R}}+\left( -1\right) ^{k+1}d_{A,\mathbb{R}}\circ
\int\nolimits_{\Delta ^{k}}\right) \Omega \right) \left( a_{0},\ldots
,a_{n-k}\right) \\
& =\sum\limits_{j=1}^{k}\left( -1\right) ^{j+1}\int\nolimits_{\Delta ^{k}}%
\frac{\partial }{\partial \tilde{t}^{\,j}}\left( \Omega \hspace{-0.1cm}%
\left( \frac{\partial }{\partial \,\tilde{t}^{1}},\ldots \widehat{j}\ldots ,%
\frac{\partial }{\partial \,\tilde{t}^{k}},a_{0},\ldots ,a_{n-k}\right)
\right) _{|\left( t,\bullet \right) }dt_{1}\ldots dt_{k} \\
& =\sum\limits_{j=1}^{k}\left( -1\right) ^{j+1}\int\nolimits_{\Delta
_{\left( \,j\right) }^{k-1}}\left. \Omega \hspace{-0.1cm}\left( \hspace{%
-0.1cm}\frac{\partial }{\partial \,\tilde{t}^{1}},\ldots \widehat{j}\ldots ,%
\frac{\partial }{\partial \,\tilde{t}^{k}},a_{0},\ldots ,a_{n-k}\hspace{%
-0.1cm}\right) \right\vert _{(t_{1},\ldots ,t_{j-1},1-\sum\limits_{\underset{%
i\neq j}{i=1}}^{k}t_{i},t_{j+1},\ldots ,t_{k},\bullet )}\hspace{-0.9cm}%
dt_{1}\ldots \widehat{j}\ldots dt_{k} \\
& +\sum\limits_{j=1}^{k}\left( -1\right) ^{j+1}\int\nolimits_{\Delta
_{\left( \,j\right) }^{k-1}}\left. \Omega \hspace{-0.1cm}\left( \frac{%
\partial }{\partial \,\tilde{t}^{1}},\ldots \widehat{j}\ldots ,\frac{%
\partial }{\partial \,\tilde{t}^{k}},a_{0},\ldots ,a_{n-k}\right)
\right\vert _{\left( t_{1},...,t_{j-1},0,t_{j+1},...,t_{k},\bullet \right) }%
\hspace{-0.65cm}dt_{1}\ldots \widehat{j}\ldots dt_{k}.
\end{align*}%
On the other hand,%
\begin{eqnarray*}
&&\left( \left( \dsum\nolimits_{j=0}^{k}\left( -1\right)
^{j}\int\nolimits_{\Delta ^{k-1}}\circ \,\left( \sigma _{j}^{k-1}\times 
\func{id}_{M}\right) ^{\ast }\right) \,\Omega \right) \left( a_{0},\ldots
,a_{n-k}\right) \\
&=&\int\nolimits_{\Delta ^{k-1}}\Omega _{\left( 1-t_{1}-\cdots
-t_{k-1},t_{1},\ldots ,t_{k-1},\bullet \right) }\left( d\sigma
_{0}^{k-1}\left( \left. \frac{\partial }{\partial t^{1}}\right\vert _{\left(
t_{1},\ldots ,t_{k-1}\right) }\right) ,\ldots \right. \\
&&\left.
\;\;\;\;\;\;\;\;\;\;\;\;\;\;\;\;\;\;\;\;\;\;\;\;\;\;\;\;\;\;\;\;\;\;\;\;\;\;%
\;\;\;\;...\;,d\sigma _{0}^{k-1}\left( \left. \frac{\partial }{\partial
t^{k-1}}\right\vert _{\left( t_{1},\ldots ,t_{k-1}\right) }\right)
,a_{0\,},\ldots ,a_{n-k}\right) \,dt_{1}\ldots dt_{k-1} \\
&&+\dsum\nolimits_{j=1}^{k}\left( -1\right) ^{j}\int\nolimits_{\Delta
^{k-1}}\Omega _{\left( t_{1},\ldots ,t_{j-1},0,t_{j},\ldots ,t_{k-1},\bullet
\right) }\left( d\sigma _{j}^{k-1}\left( \left. \frac{\partial }{\partial
t^{1}}\right\vert _{\left( t_{1},\ldots ,t_{k-1}\right) }\right) ,\ldots
\,\right. \\
&&\left.
\;\;\;\;\;\;\;\;\;\;\;\;\;\;\;\;\;\;\;\;\;\;\;\;\;\;\;\;\;\;\;\;\;\;\;\;\;\;%
\;\;\;\;...\,,d\sigma _{j}^{k-1}\left( \left. \frac{\partial }{\partial
t^{k-1}}\right\vert _{\left( t_{1},\ldots ,t_{k-1}\right) }\right)
,a_{0\,},\ldots ,a_{n-k}\right) \,dt_{1}\ldots dt_{k-1}.
\end{eqnarray*}%
In view of the fact that 
\begin{equation}
d\sigma _{0}^{k-1}\left( \left. \frac{\partial }{\partial t^{s}}\right\vert
_{\left( t_{1},\ldots ,t_{k-1}\right) }\right) =\left. \left( -\,\frac{%
\partial }{\partial \,\tilde{t}\,^{1}}+\frac{\partial }{\partial \,\tilde{t}%
\,^{s+1}}\right) \right\vert _{\left(
1-\sum\limits_{i=1}^{k-1}t_{i},t_{1},\ldots ,t_{k-1}\right) }  \label{for1}
\end{equation}%
for all $1\leq s\leq k-1$, and for $1\leq j\leq k-1$%
\begin{equation}
d\sigma _{j}^{k-1}\left( \left. \frac{\partial }{\partial t^{s}}\right\vert
_{\left( t_{1},\ldots ,t_{k-1}\right) }\right) =\left\{ 
\begin{array}{l}
\left. \frac{\partial }{\partial \,\tilde{t}\,^{s}}\right\vert _{\left(
t_{1},\ldots ,t_{j-1},0,t_{j},\ldots ,t_{k-1}\right) },\ \text{if}\ 1\leq
s<j, \\ 
\; \\ 
\left. \frac{\partial }{\partial \,\tilde{t}\,^{s+1}}\right\vert _{\left(
t_{1},\ldots ,t_{j-1},0,t_{j},\ldots ,t_{k-1}\right) },\ \text{if}\ j\leq
s\leq k-1,%
\end{array}%
\right.  \label{for2}
\end{equation}%
and using the suitable transformation we see that the second term of the
above is equal to%
\begin{equation*}
\sum_{j=1}^{k}\left( -1\right) ^{j}\int\nolimits_{\Delta _{\left( j\right)
}^{k-1}}\left. \Omega \left( \frac{\partial }{\partial \tilde{t}^{1}},\ldots 
\widehat{j}\ldots ,\frac{\partial }{\partial \,\tilde{t}^{k}},a_{0},\ldots
,a_{n-k}\right) \right\vert _{\left( s_{1},\ldots ,s_{j-1},0,s_{j+1},\ldots
,s_{k},\bullet \right) }\hspace{-0.1cm}ds_{1}\ldots \widehat{j}\ldots ds_{k}.
\end{equation*}%
This is the second term of $\left( \int\nolimits_{\Delta ^{k}}d_{T\mathbb{R}%
^{k}\times A,\mathbb{R}}\Omega +\left( -1\right) ^{k+1}d_{A,\mathbb{R}%
}\left( \int\nolimits_{\Delta ^{k}}\Omega \right) \right) \left(
a_{0},\ldots ,a_{n-k}\right) $. We apply (\ref{for1}), (\ref{for2}) again
and deduce that%
\begin{multline*}
\left( -1\right) ^{0}\left( \left( \int\nolimits_{\Delta ^{k-1}}\circ
\,\left( \sigma _{0}^{k-1}\times \func{id}_{M}\right) ^{\ast }\right) \Omega
\right) \left( a_{0},\ldots ,a_{n-k}\right) \\
=\sum_{j=1}^{k}\left( -1\right) ^{j+1}\int\nolimits_{\Delta ^{k-1}}\left.
\Omega \left( \frac{\partial }{\partial \,\tilde{t}\,^{1}},\ldots \widehat{j}%
\ldots ,\frac{\partial }{\partial \,\tilde{t}\,^{k}},a_{0},\ldots
,a_{n-k}\right) \right\vert _{(1-\sum\limits_{i=1}^{k-1}t_{i},t_{1},\ldots
,t_{k-1},\bullet )}dt_{1}\ldots dt_{k-1}.
\end{multline*}%
Now, using transformations $\phi _{1}=\func{id}_{\Delta ^{k-1}}$, $\phi
_{j}:\Delta ^{k-1}\longrightarrow \Delta ^{k-1}$,%
\begin{eqnarray*}
\phi _{2}\left( t_{1}^{\prime },\ldots ,t_{k-1}^{\prime }\right) &=&\left(
1-\sum\nolimits_{i=1}^{k-1}t_{i}^{\prime },t_{2}^{\prime },\ldots
,t_{k-1}^{\prime }\right) , \\
\phi _{j}\left( t_{1}^{\prime },\ldots ,t_{k-1}^{\prime }\right) &=&\left(
t_{2}^{\prime },\ldots ,t_{j-1}^{\prime
},1-\sum\nolimits_{i=1}^{k-1}t_{i}^{\prime },t_{j}^{\prime },\ldots
,t_{k-1}^{\prime }\right) , \\
2 &<&j\leq k-1,
\end{eqnarray*}%
after a change of variables in the integral ($|J_{\phi _{j}}|=1$) we get 
\begin{align*}
& \sum_{j=1}^{k}\left( -1\right) ^{j+1}\int\limits_{\Delta ^{k-1}}\left.
\Omega \left( \frac{\partial }{\partial \,\tilde{t}\,^{1}},\ldots \widehat{j}%
\ldots ,\frac{\partial }{\partial \,\tilde{t}\,^{k}},a_{0},\ldots
,a_{n-k}\right) \right\vert _{(1-\sum\limits_{i=1}^{k}t_{i},t_{1},\ldots
,t_{k-1},\bullet )}dt_{1}\ldots dt_{k-1} \\
& =\sum_{j=1}^{k}\left( -1\right) ^{j+1}\int\limits_{\Delta ^{k-1}}\left.
\Omega \left( \frac{\partial }{\partial \,\tilde{t}\,^{1}},\ldots \widehat{j}%
\ldots ,\frac{\partial }{\partial \,\tilde{t}\,^{k}},a_{0},\ldots
,a_{n-k}\right) \right\vert _{(t_{1},\ldots
,t_{j-1},1-\sum\limits_{i=1}^{k-1}t_{i},t_{j},\ldots ,t_{k-1},\bullet )}%
\hspace{-0.7cm}dt_{1}\ldots dt_{k-1} \\
& =\sum_{j=1}^{k}\left( -1\right) ^{j+1}\int\limits_{\Delta ^{k-1}}\left.
\Omega \left( \frac{\partial }{\partial \,\tilde{t}\,^{1}},\ldots \widehat{j}%
\ldots ,\frac{\partial }{\partial \,\tilde{t}\,^{k}},a_{0},\ldots
,a_{n-k}\right) \right\vert _{(s_{1},\ldots
,1-\sum\limits_{i=1}^{k-1}s_{i},s_{j+1},\ldots ,s_{k},\bullet )}\hspace{-1cm}%
ds_{1}\ldots \widehat{j}\ldots ds_{k},
\end{align*}%
i.e. the first term of$\,\left( \int\nolimits_{\Delta ^{k}}d_{T\mathbb{R}%
^{k}\times A,\mathbb{R}}\Omega +\left( -1\right) ^{k+1}d_{A,\mathbb{R}%
}\left( \int\nolimits_{\Delta ^{k}}\Omega \right) \right) \left(
a_{0},\ldots ,a_{n-k}\right) $. We have thus proved 
\begin{equation*}
\int\nolimits_{\Delta ^{k}}d_{T\mathbb{R}^{k}\times A,\mathbb{R}}\Omega
+\left( -1\right) ^{k+1}d_{A,\mathbb{R}}\left( \int\nolimits_{\Delta
^{k}}\Omega \right) =\dsum\nolimits_{j=0}^{k}\left( -1\right)
^{j}\int\nolimits_{\Delta ^{k-1}}\left( \sigma _{j}^{k-1}\times \func{id}%
_{M}\right) ^{\ast }\Omega \,.
\end{equation*}
\end{proof}

If we restrict the discussion to differential (linear) forms in (\ref{Stokes}%
), on the right side of (\ref{Stokes}) we obtain operators of pullback of
forms.

Let 
\begin{equation*}
\widetilde{\dint }_{\Delta ^{k}}=\left. \dint\nolimits_{\Delta
^{k}}\right\vert \Omega \left( T\mathbb{R}^{k}\times A\right) :\Omega \left(
T\mathbb{R}^{k}\times A\right) \longrightarrow \Omega \left( A\right)
\end{equation*}%
be a restriction of $\dint\nolimits_{\Delta ^{k}}$ to the module $\Omega
\left( T\mathbb{R}^{k}\times A\right) $ of differential forms on the Lie
algebroid $T\mathbb{R}^{k}\times A$. Therefore, as a corollary we obtain the
Stokes theorem for differential forms on Lie algebroids (see also \cite%
{Vaisman}).

\begin{theorem}
\label{Stokes_linear}\emph{(The Stokes theorem for differential forms on Lie
algebroids) }For every $k\in \mathbb{N}$,%
\begin{equation}
\widetilde{\dint }_{\Delta ^{k}}\circ \,d_{T\mathbb{R}^{k}\times A}+\left(
-1\right) ^{k+1}d_{A}\circ \widetilde{\dint }_{\Delta
^{k}}=\dsum\nolimits_{j=0}^{k}\left( -1\right) ^{j}\widetilde{\dint }%
_{\Delta ^{k-1}}\circ \,\left( d\sigma _{j}^{k-1}\times \func{id}_{A}\right)
^{\ast },
\end{equation}%
where $\sigma _{j}^{k}:\mathbb{R}^{k}\rightarrow \mathbb{R}^{k+1}$\ for $%
0\leq j\leq k$\ are functions defined in Theorem \ref{Stokes theorem for
R-linear forms} and $\left( d\sigma _{j}^{k-1}\times \func{id}_{A}\right)
^{\ast }:\Omega \left( T\mathbb{R}^{k}\times A\right) \rightarrow \Omega
\left( T\mathbb{R}^{k-1}\times A\right) $ is the pullback of forms via the
homomorphism of Lie algebroids $d\sigma _{j}^{k-1}\times \func{id}_{A}:T%
\mathbb{R}^{k-1}\times A\rightarrow T\mathbb{R}^{k}\times A$ over $\sigma
_{j}^{k-1}\times \func{id}_{M}$.
\end{theorem}

\section{Homotopy Operators}

\begin{definition}
\emph{\cite{Kubarski-invariant} Let }$\left( A,\rho _{A},[\![\cdot ,\cdot
]\!]_{A}\right) $\emph{\ and }$\left( B,\rho _{B},[\![\cdot ,\cdot
]\!]_{B}\right) $\emph{\ be Lie algebroids on manifolds }$M$\emph{\ and }$N$%
\emph{,\ respectively. A }homotopy\emph{\ joining two homomorphisms }$\Phi
_{0}:A\rightarrow B$\emph{, }$\Phi _{1}:A\rightarrow B$\emph{\ of Lie
algebroids is a homomorphism of Lie algebroids}%
\begin{equation*}
\Phi :T\mathbb{R}\times A\longrightarrow B
\end{equation*}%
\emph{with }$\Phi \left( \theta _{0},\cdot \right) =\Phi _{0}$\emph{\ and }$%
\Phi \left( \theta _{1},\cdot \right) =\Phi _{1}$\emph{, where} $\theta
_{0}\in T_{0}\mathbb{R}$ \emph{and} $\theta _{1}\in T_{1}\mathbb{R}$ \emph{%
are null vectors;} $T\mathbb{R}\times A$ \emph{denotes the Cartesian product
of Lie algebroids }$T\mathbb{R}$ \emph{and }$A$\emph{. If there exists a
homotopy joining two homomorphisms, we say that these homomorphisms are }%
homotopic.
\end{definition}

As a corollary from Theorem \ref{Stokes_linear} we obtain the following
result which is a generalization of the result for regular Lie algebroids
from \cite{Kubarski-invariant}.

\begin{theorem}
\label{operatorhomotopii}\emph{\ }Let $\left( A,\rho _{A},[\![\cdot ,\cdot
]\!]_{A}\right) $ and $\left( B,\rho _{B},[\![\cdot ,\cdot ]\!]_{B}\right) $
be Lie algebroids on a manifold $M$ and $\Phi _{0}:A\rightarrow B$, $\Phi
_{1}:A\rightarrow B$ homomorphisms of Lie algebroids. If $\Phi :T\mathbb{R}%
\times A\rightarrow B$ is a homotopy joining $\Phi _{0}$ to $\Phi _{1}$, then%
\begin{equation*}
h=\int\nolimits_{\Delta ^{1}}\circ \,\,\Phi ^{\ast }:\Omega \left( B\right)
\longrightarrow \Omega \left( A\right)
\end{equation*}%
is a chain operator joining $\Phi _{0}^{\ast }:\Omega \left( B\right)
\rightarrow \Omega \left( T\mathbb{R}\times A\right) $ to $\Phi _{1}^{\ast
}:\Omega \left( B\right) \rightarrow \Omega \left( T\mathbb{R}\times
A\right) $, i.e. 
\begin{equation*}
h\circ d_{B}+d_{A}\circ h=\Phi _{1}^{\ast }-\Phi _{0}^{\ast }.
\end{equation*}
\end{theorem}

\begin{proof}
Since for $j\in \left\{ 1,2\right\} $, $\Phi _{j}=\Phi \left( \theta
_{j},\cdot \right) $, $\left( d\sigma _{j}^{0}\times \func{id}_{A}\right)
^{\ast }\circ \Phi ^{\ast }=\left( \Phi \circ d\sigma _{j}^{0}\times \func{id%
}_{A}\right) ^{\ast }$ where $\sigma _{0}^{0}$,$\,\sigma _{1}^{0}:\left\{
0\right\} \rightarrow \Delta ^{1}=\left[ 0,1\right] $\ are functions defined
by $\sigma _{0}^{0}\left( 0\right) =1$, $\sigma _{1}^{0}\left( 0\right) =0$
and $\left( \Phi \circ d\sigma _{j}^{0}\times \func{id}_{A}\right) ^{\ast
}\left( 0\times a\right) =\Phi _{1-j}\left( a\right) $ for all $a\in \Gamma
\left( A\right) $, we obtain%
\begin{equation*}
\int\nolimits_{\Delta ^{0}}\circ \left( d\sigma _{j}^{0}\times \func{id}%
_{A}\right) ^{\ast }\circ \Phi ^{\ast }=\Phi _{1-j}^{\ast }.
\end{equation*}%
From the above, the Stokes formula (Theorem \ref{Stokes_linear} for $k=1$)
and the commutativity of a pullback of differential forms on the Lie
algebroid via a homomorphism of Lie algebroids with differentials, we
conclude that: 
\begin{eqnarray*}
\Phi _{1}^{\ast }-\Phi _{0}^{\ast } &=&\left( \Phi \circ d\sigma
_{0}^{0}\times \func{id}_{A}\right) ^{\ast }-\left( \Phi \circ d\sigma
_{1}^{0}\times \func{id}_{A}\right) ^{\ast } \\
&=&\left( \left( d\sigma _{0}^{0}\times \func{id}_{A}\right) ^{\ast }-\left(
d\sigma _{j}^{0}\times \func{id}_{A}\right) ^{\ast }\right) \circ \Phi
^{\ast } \\
&=&\left( \int\nolimits_{\Delta ^{1}}\circ \,d_{T\mathbb{R}\times A}+\left(
-1\right) ^{1+1}d_{A}\circ \int\nolimits_{\Delta ^{1}}\right) \circ \Phi
^{\ast } \\
&=&\int\nolimits_{\Delta ^{1}}\circ \,\,\left( d_{T\mathbb{R}\times A}\circ
\Phi ^{\ast }\right) +d_{A}\circ \left( \int\nolimits_{\Delta ^{1}}\circ
\,\,\Phi ^{\ast }\right) \\
&=&\left( \int\nolimits_{\Delta ^{1}}\circ \,\,\Phi ^{\ast }\right) \circ
d_{B}+d_{A}\circ \left( \int\nolimits_{\Delta ^{1}}\circ \,\,\Phi ^{\ast
}\right) \\
&=&h\circ d_{B}+d_{A}\circ h.
\end{eqnarray*}
\end{proof}

\begin{remark}
\emph{The projection on the second factor} $\pi :T\mathbb{R}^{k}\times
A\rightarrow A$ \emph{is a homomorphism of Lie algebroids over} $\func{pr}%
_{2}$\emph{. Moreover,} $G_{0}:A\rightarrow T\mathbb{R}^{k}\times A$, $%
G_{0}\left( a\right) =\left( \Theta _{0},a\right) $\emph{, where} $\Theta
_{0}\in T_{0}\mathbb{R}^{k}$ \emph{is the null vector tangent to} $\mathbb{R}%
^{k}$ \emph{at the zero point, is a homomorphism of Lie algebroids over }$%
j_{0}:M\rightarrow \mathbb{R}^{k}\times M$, $j_{0}\left( x\right) =\left(
0,x\right) $\emph{. Since} $\pi \circ G_{0}=\limfunc{id}\nolimits_{A}$\emph{%
, then }$G_{0}^{\ast }:\Omega ^{\bullet }\left( T\mathbb{R}^{k}\times
A\right) \rightarrow \Omega ^{\bullet }\left( A\right) $\emph{\ induces the
homomorphism }$G_{0}^{\#}:H^{\bullet }\left( T\mathbb{R}^{k}\times A\right)
\rightarrow H^{\bullet }\left( A\right) $ \emph{in cohomology such that}%
\begin{equation}
G_{0}^{\#}\circ \pi ^{\#}=\limfunc{id}\nolimits_{H^{\bullet }\left( A\right)
}.  \label{G_0}
\end{equation}%
\emph{Consider }$f:\mathbb{R}\times \mathbb{R}^{k}\rightarrow \mathbb{R}^{k}$%
, $f\left( s,t\right) =s\cdot t$\emph{. Since} $df:T\left( \mathbb{R}\times 
\mathbb{R}^{k}\right) =T\mathbb{R}\times T\mathbb{R}^{k}\rightarrow T\mathbb{%
R}^{k}$\emph{\ is a homomorphism of Lie algebroids over }$f$\emph{, then }$%
\Phi =df\times \limfunc{id}\nolimits_{A}:T\mathbb{R}\times \left( T\mathbb{R}%
^{k}\times A\right) \rightarrow T\mathbb{R}^{k}\times A$\emph{\ is a
homomorphism of Lie algebroids over }$f\times \limfunc{id}\nolimits_{M}$%
\emph{\ which is a homotopy joining} $G_{0}\circ \pi $\emph{\ to} $\limfunc{%
id}\nolimits_{T\mathbb{R}\times A}$\emph{. According to Theorem \ref%
{operatorhomotopii}, we conclude that there exists a chain operator }$%
h:\Omega \left( A\right) \rightarrow \Omega \left( T\mathbb{R}^{k}\times
A\right) $ \emph{joining} $\left( G_{0}\circ \pi \right) ^{\ast }=\pi ^{\ast
}\circ G_{0}^{\ast }$\emph{\ to }$\limfunc{id}\nolimits_{T\mathbb{R}%
^{k}\times A}^{\ast }$, 
\begin{equation*}
h\circ d_{T\mathbb{R}^{k}\times A}+d_{T\mathbb{R}^{k}\times A}\circ h=%
\limfunc{id}\nolimits_{T\mathbb{R}^{k}\times A}^{\ast }-\pi ^{\ast }\circ
G_{0}^{\ast }.
\end{equation*}%
\emph{Therefore }$\limfunc{id}\nolimits_{H^{\bullet }\left( T\mathbb{R}%
^{k}\times A\right) }-\pi ^{\#}\circ G_{0}^{\#}$\emph{\ is the zero-map in
cohomology. From this and (\ref{G_0}) we deduce that}%
\begin{equation*}
H^{\bullet }\left( A\right) \cong H^{\bullet }\left( T\mathbb{R}^{k}\times
A\right) .
\end{equation*}
\end{remark}

\end{document}